\documentclass[a4paper,reqno,12pt]{article}
\usepackage{mathrsfs}

\usepackage{amsmath}
\usepackage{amsfonts}
\usepackage{amssymb}
\usepackage{amsthm}
\usepackage[mathscr]{eucal}
\usepackage[hmargin=2.5cm,vmargin=2.5cm]{geometry}

\usepackage{graphicx}
\usepackage{indentfirst}
\usepackage{newlfont}

\numberwithin{equation}{section}

\newcommand{\real}{\mathbb{R}}

\theoremstyle{plain}
\newtheorem{theorem}{Theorem}[section]
\newtheorem{lemma}[theorem]{Lemma}

\newtheorem{definition}[theorem]{Definition}
\newtheorem{proposition}[theorem]{Proposition}

\theoremstyle{remark}
\newtheorem{remark}[theorem]{Remark}

\begin{document}

\title{
Dynamics of stochastic non-Newtonian fluids
driven by fractional Brownian motion with Hurst parameter $H \in (\frac{1}{4},\frac{1}{2})$
\footnote{Support by NSFC (No. 10971225) and SRF for ROCS, SEM, China}}
\author{Jin Li $^a$ and Jianhua Huang\\
Department of Mathematics, National University of Defense Technology,\\
Changsha 410073,  P.R.China \\
$^a$ \textsl{\small Corresponding author E-mail: lijinmath@nudt.edu.cn} }

\date{\today}

\maketitle

\abstract{
In this paper we consider
the Stochastic isothermal, nonlinear, incompressible bipolar viscous fluids
driven by a genuine cylindrical fractional Bronwnian motion with Hurst parameter $H \in (\frac{1}{4},\frac{1}{2})$
under Dirichlet boundary condition on 2D square domain.
First we prove the existence and regularity of the stochastic convolution corresponding
to the stochastic non-Newtonian fluids.
Then we obtain the existence and uniqueness results for the stochastic non-Newtonian fluids.
Under certain condition, the random dynamical system generated by non-Newtonian fluids
has a random attractor.
}



\section{Introduction}
In this paper, the stochastic non-Newtonian fluids driven by fractional Bronwian motion (fBm, for short)
on $[0,\pi] \times [0,\pi]$ are studied.
The constitutive relations for such fluids were introduced by Bellout, Bloom and Ne\v{c}as \cite{Bellout Bloom Necas}
to describe the isothermal, nonlinear, incompressible bipolar viscous fluid.
It has the form
\begin{align}
    \tau_{ij}  & = -p \delta_{ij}
      +2\mu_0 (\epsilon + |e|^2 )^{-\frac{\alpha}{2}} e_{ij} - 2 \mu_1 \triangle e_{ij}
                 \label{Constitutive relation: BBN 1} \\
    \tau_{ijk} & = 2\mu_1 \frac{\partial e_{ij}}{\partial x_k}  \label{Constitutive relation: BBN 2}
\end{align}
where $\tau_{ij}$ is the components of the stress tensor,
$\tau_{ijk}$ is the components of the first multipolar stress tensor,
and $p$ is the pressure.
$e_{ij}$ are the components of the rate of deformation tensor, i.e.
\begin{equation}
    e_{ij} = \frac{1}{2} \left( \frac{\partial u_i}{\partial x_j}
        + \frac{\partial u_j}{\partial x_i} \right) .
\end{equation}
$\epsilon, \mu_0, \mu_1 >1$ and $\alpha, 0< \alpha \leq 1$, are constitutive parameters.
The constitutive relation \eqref{Constitutive relation: BBN 1} and \eqref{Constitutive relation: BBN 2},
and the condition of incompressibility,
yield the following nonlinear partial differential equations (we call it Bellout-Bloom-Ne\v{c}as fluids):
\begin{align}
     \rho \frac{\partial{u}}{\partial{t}} +
     \rho (u \cdot \nabla)u + \nabla p
    & = \nabla \cdot (\mu (u)e-2\mu_1 \Delta e)
       + \rho f  \label{BBN model 1}\\
    \nabla \cdot u &=0   \label{BBN model 2} ,
\end{align}
where $\rho$ is the constant density,
$\mu (u)=2\mu_0 (\epsilon + |e|^2)^{-\frac{\alpha}{2}}$ is a nonlinear viscosity,
and $f$ is the external body force vector.

There are many works concerning existence and regularity of solution to the Bellout-Bloom-Ne\v{c}as fluids
and its dynamics (see, for instance,
\cite{Ladyzhenskaya viscousflow, Bellout Bloom Necas, Bloom Hao existence, Zhao Duan, Zhao Zhou, GuoGuo}).
In this paper, we consider the stochastic Bellout-Bloom-Ne\v{c}as fluids driven by
a genuine cylindrical fBm
with Hurst parameter $H \in (\frac{1}{4},\frac{1}{2})$:
\begin{align}\label{Problem: initial}
     \frac{\partial{u}}{\partial{t}} + (u \cdot \nabla)u + \nabla p
    & = \nabla \cdot (\mu (u)e-2\mu_1 \Delta e)+ \frac{dB^H(t)}{dt}
     \quad & x \in \mathcal{O}, \ t>0  \\
    \nabla \cdot u &=0,
     \quad & x \in \mathcal{O}, \ t>0  \\
     u &= 0, \ \tau_{ijk} \eta_j \eta_k = 0,
     \quad & x \in \partial \mathcal{O}, \ t \geq 0  \\
     u &= u_0, \quad &x \in \mathcal{O}, t=0  \label{Problem: initial end}
\end{align}
where $\mathcal{O}$ is a 2D square,
i.e., $\mathcal{O} = \{(x_1,x_2) \ | \ 0 < x_1 < \pi, \ 0 < x_2 < \pi \}$.
The fractional Brownian noise enters linearly in the equation
and the fBm models the noise source.
The fBm is a family of Gaussian processes
and some useful properties of these process were given by Mandelbrot and Van Ness \cite{Mandelbrot VanNess}.
For $H < \frac{1}{2}$ the fBm is not a semimartingale
and they can be used in modeling phenomena with
intermittency and anti-persistence such as financial turbulence.

The preprint \cite{Li Huang} treats stochastic Bellout-Bloom-Ne\v{c}as fluids that
the noise term has a trace-class correlation,
and moreover they treat the case $H > \frac{1}{2}$,
which allows one to solve the equation using stochastic integrals understood in a pathwise way.
In this paper we first provide a detailed study of the existence and regularity properties of
the stochastic convolution corresponding to stochastic non-Newtonian fluids
driven by fBm with $H \in (\frac{1}{4},\frac{1}{2})$.
The approach for dropping the Hilbert-Schmidt operator hypothesis follows \cite{Tindel Tudor Viens}.
Then we establish the existence of solution by a modified version of fixed point theorem
in some specific Banach space.
The a priori estimate for solution in intersection space follows \cite{Li Huang}.
Finally we construct the random dynamical system associated by non-Newtonian fluids and prove
the existence of random attractor follows the framework of \cite{Crauel Flandoli}.

We emphasize four points in our paper.
(i) By careful estimation, the growth speed for eigenvalues of the linear differential operator
ensures the convergence of the stochastic convolution respect to fBm integral.
When investigating the dynamics of Bellout-Bloom-Ne\v{c}as equation
perturbed by fraction Brownian noise,
(ii) the ergodic property of stochastic integral with respect to fBm ensures us to
construct a random dynamical system,
(iii) the at most polynomial growth of sample path ensures
the existence of random absorbing sets,
and the regularity of stochastic convolution with respect to infinite dimensional fBm
with $H \in (\frac{1}{4}, \frac{1}{2})$ ensures
the existence of absorbing set with higher regularity.
(iv) Under certain condition
( which limits the injection parameters of fractional integral space to $L^2$ space ),
the random dynamical system has a random attractor as the  case $H > \frac{1}{2}$ in \cite{Li Huang}.

The rest of the paper is organized as follow.
In section 2, we formulate the mathematical setting for
stochastic Bellout-Bloom-Ne\v{c}as fluids with Dirichlet boundary condition on 2D square domain
and recall the Wiener integrals with respect to infinite dimensional fBm
by the framework of \cite{Tindel Tudor Viens}.
Section 3 is devoted to the existence and regularity properties of
the stochastic convolution.
In section 4, the global existence and uniqueness of solution is obtained.
In section 5, we prove the existence of a random attractor
for the random dynamical system generated by non-Newtonian fluids.

\section{Preliminaries}


We use the standard mathematical framework of this model.

$\mathcal{V} = \left\{ \phi = (\phi_1, \phi_2) \in \left( C_0^{\infty} \left( \mathcal{O} \right) \right)^2 :
\ \nabla \cdot \phi =0 \text{ and } \phi=0 \text{ on } \partial \mathcal{O} \right\}$.

$H = $ the closure of $ \mathcal{V}$ in $\left( L^2\left(\mathcal{O}\right)\right)^2$ with norm $|\cdot|$.

$V = $ the closure of $\mathcal{V}$ in  $\left( H^2(\mathcal{O})\right)^2$ with norm $|\cdot|_V$.

$\dot{H}^k =  $ the closure of $\mathcal{V}$ in $\left( H^k(\mathcal{O})\right)^2$ with
norm $|\cdot|_k$ for $k \in \mathbb{N}$.


For other Banach space $Y$, denote the norm $|\cdot|_Y$.

$B_Y(M):= \{ |x|_Y \leq M, x \in Y \}$.

Let $H'$ and $V'$ be the dual spaces of $H$ and $V$ respectively.
It follows that $V \subset H \equiv H' \subset V'$ and the injections are continuous.

Denote by $(\cdot, \ \cdot)$ the inner product in $H$,
$<\cdot, \ \cdot>$ the dual pair between $V'$ and $V$.

$\mathcal{L}(H,V)$ := $\{$ all the bounded linear operator from  $H$ to $V$ $\}$,
$\mathcal{L}(H) := \mathcal{L}(H,H)$.

$\mathcal{L}_2(H)$ := $\{$ all the Hilbert-Schmidt operator from  $H$ to $H$ $\}$.

We use lowercase $c_i, i\in \mathbb{N}$ for global constants and
capital $C$ for local constants which may change value from line to line.

Inspired by \cite{Lions Magenes}, we select space $X := C([0,T];H) \bigcap L^2(0,T;V)$
with norm $|\cdot|_X := |\cdot|_{C([0,T];H)} + |\cdot|_{L^2(0,T;V)}$ for solutions.
For the completeness of space $X$ we refer to \cite{Li Huang}.

First we define a bilinear form $a(\cdot, \ \cdot):V \times V \rightarrow \mathbb{R}$,
\begin{equation}
    a(u,v)=\frac{1}{2}(\triangle u, \triangle v).
\end{equation}

\begin{proposition}\label{Proposition: a}
The rate of deformation tensor $e(u)$ and bilinear form $a(\cdot, \ \cdot)$ have the following properties:
  \begin{enumerate}
    \item[(i)]
    \begin{align}
        \nabla \cdot e(u) &= \frac{1}{2} \triangle u \quad \forall u \in V, \\
        \nabla \cdot (\triangle e(u)) &= \frac{1}{2}\triangle^2 u
        \quad \forall u \in \dot{H}^4 .      \label{Equation: 1}
    \end{align}
    \item[(ii)] For $ \ u,v \in V$,
    \begin{equation}
        \frac{1}{2}(\triangle u, \triangle v)
        = 2 \sum_{i,j,k=1}^{2} \int \frac{\partial e_{ij}(u)}{\partial x_j}
            \cdot \frac{\partial e_{ik}(u)}{\partial x_k} dx
        =\sum_{i,j,k=1}^{2} \int \frac{\partial e_{ij}(u)}{\partial x_k}
            \cdot \frac{\partial e_{ij}(u)}{\partial x_k} dx .
    \end{equation}
    \item[(iii)] (\cite{Bloom Hao existence} Lemma 2.3) There exist $ \ c_1, c_2>0$, s.t.
    \begin{equation}
        c_1|u|_V^2 \leq a(u,u) \leq c_2 |u|_V^2, \quad \forall u \in V .
    \end{equation}
  \end{enumerate}
\end{proposition}

Next we define the abstract differential operator $A$ and the analytic semigroup $S(\cdot)$.
According to (iii) of Proposition \ref{Proposition: a},
we can use Lax-milgram Theorem to define
$A \in \mathcal{L}(V,V')$:
\begin{equation}
    <Au,v>=a(u,v) \quad \forall \ u,v \in V.
\end{equation}

And we have
\begin{proposition}\label{Proposition: A1}
  \cite{Li Huang}
  \begin{enumerate}
    \item[(i)] Operator $A$ is an isometric form $V$ to $V'$.
    Furthermore, let $D(A)=\{ u \in V : a(u,v)=(f,v), f \in H \}$.
    Then $A \in \mathcal{L}(D(A),H)$ is an isometric form $D(A)$ to $H$.

    \item[(ii)] Operator $A$ is self-adjoint positive with compact inverse.
    By Hilbert Theorem, there exist eigenvectors $\{ e_i \}_{i=1}^{\infty} \subset D(A)$
    and eigenvalues $\{ \lambda_i \}_{i=1}^{\infty}$ s.t.
    \begin{align}
      & Ae_i=\lambda e_i, \quad e_i \in D(A), \quad i=1,2, \cdots \\
      & 0<\lambda_1 \leq \lambda_2 \leq \cdots \leq\lambda_i \leq \cdots,
      \quad \lim_{i\rightarrow \infty} \lambda_i = \infty.
    \end{align}
    And $\{ e_i \}_{i=1}^{\infty}$ form an orthonormal basis for $H$.

    \item[(iii)] For $ u \in D(A)$,
    \begin{equation}
        Au=\nabla \cdot (\triangle e(u))=\frac{1}{2}\triangle^2 u,
    \end{equation}
    i.e., $A=P\triangle^2$ where $P$ is the Leray projection operator from $L^2(\mathcal{O})$ to $H$.
  \end{enumerate}
\end{proposition}

Noticing that $A$ is a self-adjoint positive linear operator with discrete spectrum,
we define the fractional power of $A$ by following the framework of
Chueshov \cite{Chueshov} section 2.1:
\begin{definition}
  For $ \alpha >0 $,
  \begin{align}
    D(A^{\alpha}) & := \left\{ h=\sum_{k=1}^{\infty}c_k e_k \in H:
    \sum_{k=1}^{\infty}c_k^2 (\lambda_k^{\alpha})^2 <\infty  \right\}  , \\
    D(A^{-\alpha}) & := \left\{ \text{formal serial }\sum c_k e_k
    \text{ such that } \sum_{k=1}^{\infty}c_k^2 (\lambda_k^{\alpha})^2 <\infty \right\}  , \\
    A^{\alpha}h  & =\sum_{k=1}^{\infty} c_k \lambda_{k}^{\alpha}e_k, \quad h\in D(A^{\alpha}).
  \end{align}
\end{definition}

Let $\mathscr{F}_{\alpha} \equiv D(A^{\alpha})$.
Then $\mathscr{F}_{\alpha}$ is a separable Hilbert space with
the inner product $(u,v)_{\mathscr{F}_{\alpha}}=(A^{\alpha}u, A^{\alpha}v)$
and the norm $||u||_{\mathscr{F}_{\alpha}}=|A^{\alpha}u|$;
$\mathscr{F}_{-\alpha}$ denotes the union of bounded linear functional on $\mathscr{F}_{\alpha}$;
Particularly, we have $\mathscr{F}_0 = H$, $\mathscr{F}_{1/2} = V$,
$\mathscr{F}_{-1/2} = V'$ and the norm of $V$ and $H^2(\mathcal{O})$ are equivalent;
For $ \sigma_1 > \sigma_2$,
the space $\mathscr{F}_{\sigma_1}$ is compactly embedded into $\mathscr{F}_{\sigma_2}$.

Since $A$ is a densely defined self-adjoint bounded-below operator in Hilbert space $H$,
we deduce that $A$ is a sector operator and it generates an analytic semigroup
$S \in \mathcal{L}(H)$ (see, for instance, \cite{DanielHenry} section 1.3),
\begin{equation}
    S(t):=e^{-tA}=\int_{0}^{\infty} e^{-t\lambda} d E_{\lambda}.
\end{equation}

For a survey of the properties of the analytic semigroup we refer to \cite{Li Huang}.

In order to write down the abstract evolution equation, we need to handle the nonlinear terms.
Following the method in dealing with Navier-Stokes equation, we define the trilinear form:
\begin{equation}
    b(u,v,w)=\sum_{i,j=1}^{2} \int_{\mathcal{O}} u_i \frac{\partial v_j}{\partial x_i} w_j dx
    \quad \forall u,v,w \in H_0^1 (\mathcal{O}).
\end{equation}
Since $V \subset H_0^1(\mathcal{O})$ is a closed subspace,
$b(\cdot,\cdot,\cdot)$ is continuous in $V\times V \times V$.
From \cite{RogerTemam} we have,
\begin{equation}\label{Equation: orthogonal for b}
    b(u,v,w)=-b(u,w,v), \quad b(u,v,v)=0 \quad \forall u,v,w,\in H_0^1(\mathcal{O}).
\end{equation}
For $ u,v \in V$, define the functional $B(u,v) \in V'$:
\begin{equation}
    <B(u,v),w>=b(u,v,w) \quad \forall w \in V
\end{equation}
and denote $B(u):=B(u,u)\in V'$.

For $ u \in V$, define $N(u)$ as
\begin{equation}
    <N(u),v>=\int_{\mathcal{O}} \mu(u) e_{ij}(u) e_{ij}(v) dx
    \quad \forall v \in V.
\end{equation}
Then $N(\cdot)$ is a continuous from $V$ to $V'$ and
\begin{equation}\label{Equation: def of N}
    <N(u),v>=- \int_{\mathcal{O}} \left( \nabla \cdot \left(\mu (u) e(u)\right) \right)
    \cdot v dx.
\end{equation}

Comprehensively, we have the following abstract evolution equation from problem
\eqref{Problem: initial}-\eqref{Problem: initial end}:
\begin{equation}\label{Problem: differential form}
  \left\{
  \begin{split}
    du+ \left( 2\mu_1 Au +B(u)+N(u)  \right) dt &= dB^H(t),  \\
    u(0) & =u_0.
  \end{split}  \right.
\end{equation}
Since the derivative of fBm exists almost nowhere,
we will give a mathematical interpret of above equation in the next subsection.
Without loss of generality, we set $\mu_1 = 1$ in the sequel.

Next we introduce the Wiener-type stochastic integral with respect to fBm.
For all $T>0$, let $\beta^H(t)$ be the one-dimensional fBm.
In this paper we only consider the case $H < \frac{1}{2}$ following the \cite{Tindel Tudor Viens}
and for a survey of Winer-type stochastic integral we refer to \cite{Biagini Hu Oksendal Zhang}.
By definition $\beta^H$ is a centered Gaussian process with covariance
\begin{equation}
    R(t,s)=E(\beta^H(t) \beta^H(t))= \frac{1}{2} (t^{2H}+s^{2H}-|t-s|^{2H}).
\end{equation}
And $\beta^H$ has the following Wiener integral representation:
\begin{equation}
    \beta^H(t) = \int_0^t K^H(t,s) d W(s),
\end{equation}
where $W$ is a Wiener process, and $K^H(t,s)$ is the kernel given by
\begin{equation}\label{Equation: defi of kernel}
    K^H(t,s)
    = c_H \left( \frac{t}{s} \right)^{H-\frac{1}{2}}  (t-s)^{H- \frac{1}{2}}
      + s^{\frac{1}{2}-H} F(\frac{t}{s}).
\end{equation}
$c_H$ is a constant and
\begin{equation}
    F(z) = c_H \left( \frac{1}{2} - H \right)
             \int_0^{z-1} r^{H - \frac{3}{2}} \left( 1-(1+r)^{H- \frac{1}{2}} \right) dr.
\end{equation}
By \eqref{Equation: defi of kernel} we obtain
\begin{equation}
    \frac{\partial K^H}{\partial t}(t,s)
      = c_H (H-\frac{1}{2}) (t-s)^{H-\frac{3}{2}} \left( \frac{s}{t}\right)^{\frac{1}{2}-H}.
\end{equation}
Denote by $\mathscr{E}_H$ the linear space of step function of the form
\begin{equation}
    \varphi(t) = \sum_{i=1}^{n} a_i 1_{(t_i,t_{i+1} ]} (t)
\end{equation}
where $n \in \mathbb{N}$, $a_i \in \real$ and by $\mathscr{H}$ the closure of $\mathscr{E}_H$
with respect to the scalar product
\begin{equation}
    <1_{[0,t]}, 1_{[0,s]}>_{\mathscr{H}} = R(t,s)
\end{equation}
For $\varphi \in \mathscr{E}_H$ we define its Weiner integral
with respect to the fBm as
\begin{equation}
    \int_0^T \varphi(s) d \beta^H(s) = \sum_{i=1}^{n} a_i (\beta^H_{t_{i+1}} - \beta^H_{t_{i}}).
\end{equation}
The mapping
\begin{equation}
    \varphi = \sum_{i=1}^{n} a_i 1_{(t_i,t_{i+1} ]}
      \rightarrow \int_0^T \varphi(s) d \beta^H(s)
\end{equation}
is an isometry between $\mathscr{E}_H$ and the linear space $span \{ \beta^H(t), 0\leq t \leq T \} $
viewed as a subspace of $L^2(0,T)$
and it can be extended to an isometry between $\mathscr{H}$
and the $\overline{span}^{L^2} \{ \beta^H(t), 0\leq t \leq T  \}$.
The image on an element $\Psi \in \mathscr{H}$ by this isometry is called the Wiener integral of
$\Psi$ with respect to $\beta^H$.

For every $s<t$, consider the operator $K^*$
\begin{equation}\label{Equation: K* representation}
    (K^*_t) \varphi (s) = K(t,s) \varphi(s) +
      \int_s^t (\varphi(r)-\varphi(s)) \frac{\partial K}{\partial r}(r,s) dr.
\end{equation}
We refer to \cite{Alos Mazet Nualart} for the proof of the fact that $K^*$ is a isometry between
$\mathscr{H}$ and $L^2(0,T)$.
For $H< \frac{1}{2}$, the reproducing kernel Hilbert space $\mathcal{H}$
can be represented by the fractional integral space.
Namely,
\begin{equation}
    \mathcal{H} = (K_H^*)^{-1} (L^2 (0,T))= I_{T^-}^{\frac{1}{2}-H}(L^2(0,T)),
\end{equation}
where $I_{T^-}^{\frac{1}{2}-H(L^2(0,T))}$ is the family of functions $f$ that
can be represented as a fractional $I_{T^-}^{\frac{1}{2}-H}$-integral
of some function $\phi \in L^2(0,T)$.
As a consequence, we have the following relationship between the Wiener integral with respect
to fBm and the Wiener integral with respect to the Wiener process:
\begin{equation}\label{Equation: fBm represetation by Bm}
    \int_0^t \varphi(s)d\beta^H(s) = \int_0^t (K^*_t \varphi)(s)dW(s)
\end{equation}
for every $t\leq T$ and $\varphi \in \mathscr{H}$ if and only if $K^*_t \varphi \in L^2(0,T)$.
Since we work only with Wiener integral over Hilbert space,
we have that if $u \in L^2(0,T;H)$ is a deterministic function,
then the relation \eqref{Equation: fBm represetation by Bm} holds
and the Wiener integral on the righthand side being well defined in $L^2(\Omega;H)$
if $K^* u$ belongs to $L^2(0,T;H)$.

In the following we concern the infinite dimensional fBm and stochastic integration
(see, e.g. \cite{Duncan Maslowski Duncan}).
A standard cylindrical fractional Brownian motion is defined now.
\begin{definition}
  \cite{Duncan Maslowski PasikDuncan} Let $(\Omega, \mathcal{F},P)$ be a complete probability space.
  A cylindrical process $<B^H,\cdot> : \Omega \times \mathbb{R}_+ \times H \rightarrow \mathbb{R}$
  on $(\Omega, \mathcal{F},P)$ is called a standard cylindrical fractional Brownian motion with the
  Hurst parameter $H \in (0,1)$ if
  \begin{enumerate}
    \item for each $x \in H \backslash \{0\}$, $\frac{1}{||x||} <B^H(\cdot),x>$ is a standard
      scalar fBm with Hurst parameter $H$;
    \item for $\alpha,\beta \in \real$ and $x,y \in H$,
    \begin{equation}
        <B^H(t), \alpha x + \beta y> = \alpha <B^H(t),x> + \beta <B^H(t),y> \quad \text{P-a.s.}
    \end{equation}
  \end{enumerate}
\end{definition}
   For $H=\frac{1}{2}$, this definition is the usual one for a standard cylindrical Wiener process in $H$.
   For the complete orthonormal basis $\{e_n\}_{n\in \mathbb{N}}$ of $H$
   (which is generated by linear differential operator A),
   letting $\beta_n^H(t) = <B^H(t),e_n>$ for $n \in \mathbb{N}$,
   the sequence of scalar processes $\{ \beta^H_n\}_{n\in \mathbb{N}}$ is independent and
   $B^H$ can be represented by the formal series
   \begin{equation}\label{Equation: series repre BH}
    B^H(t)= \sum_{n=1}^{\infty} \beta^H_n(t) e_n
   \end{equation}
   that does not converge a.s. in $H$.
   Although for any fixed $t$ the series \eqref{Equation: series repre BH} is not convergent in
   $L^2(\Omega \times H)$,
   we can always consider a Hilbert space $U_1$ such that $H \subset U_1$ such that this inclusion
   is a Hilbert-Schmidt operator.
   In this way $B^H$ given by \eqref{Equation: series repre BH} is a well-defined
   $U_1$-valued Gaussian stochastic process.

Let $\Phi(s),0 \leq s \leq T$ be a deterministic function with values in $\mathcal{L}_2(H)$,
the space of Hilbert-Schmidt operators on $H$.
The stochastic integral of $\Phi$ with respect to $B^H$ is defined by
\begin{equation}
    \int_0^t \Phi(s)dB^H(s)
     = \sum_{n=1}^{\infty} \int_0^t \Phi(s)e_n d\beta^H_n(s)
     = \sum_{n=1}^{\infty} \int_0^t (K^*(\Phi e_n))(s) d\beta_n(s)
\end{equation}
where $\beta_n$ is the standard Brownian motion.
However, as we are about to see, the stochastic linear additive equation in its mild form can
have a solution even if $\int_0^t \Phi(s)dB^H(s)$ is not properly defined as a $H$-valued process.

\section{Linear stochastic evolution equations and stochastic convolution with fBm}


In this section, we will work with a cylindrical fBm $B^H$
on the real separable Hilbert space $H$.
First we prove the existence of solution for linear stochastic evolution equation
driven by fBm with $H \in (\frac{1}{4},\frac{1}{2})$.
Consider the equation
\begin{equation}\label{1}
    dZ = AZdt + dB^H, \quad Z(0)=u_0 \in H
\end{equation}
As noted in \cite{Tindel Tudor Viens} and \cite{Duncan Maslowski PasikDuncan},
the stochastic integral $\int_{0}^{t} I_d dB^H (s)$ is not well-defined as a $H$-valued random variable
since the identity operator $I_d \notin \mathcal{L}_2(H)$.
We then consider the mild form of the equation, whose unique solution, if it exists,
can be written in the evolution form
\begin{equation}\label{Equation: mild solu of OU}
    Z(t)=S(t)u_0 + \int_{0}^{t} S(t-s) dB^H(s).
\end{equation}

\begin{remark}
  In \cite{Duncan Maslowski PasikDuncan} they consider the noise term as $\Phi d B^H$
  and assume that $\Phi \in \mathcal{L}_2$ or $S(t)\Phi \in \mathcal{L}_2$.
  In such case the infinite dimensional fBm noise and be viewed as finite dimensional
  since the Hilbert-Schmidt operator is compact.
  By contrast the noise we consider (i.e., $\Phi = I_d $) is rougher
  and there is no reason to assume that $\Phi \in \mathcal{L}_2$.
\end{remark}

In order to obtain the existence of the stochastic convolution above,
we need the follow estimate about the spectrum of operator A.

\begin{lemma}
  The eigenvalues of operator A satisfy
\begin{equation}\label{3}
    \lambda_{mn} \geq (m^2 + n^2)^2 , \quad m,n \in \mathbb{N}.
\end{equation}
\end{lemma}

\begin{proof}
  It is well known that, the eigenfunctions and eigenvalues of Dirichlet-Lapacian on a
  2D square (e.g. see \cite{Courant Hilbert}) are
  \begin{equation}\label{4}
      \left\{
        \begin{split}
          \phi(x_1,x_2) &=C \sin mx_1 \sin nx_2 \\
          \gamma_{mn} &= m^2 + n^2, \quad \text{ with } m,n \in \mathbb{N}.
        \end{split}
      \right.
  \end{equation}
Rewrite the the index of $\gamma$ and we get the spectrum
  \begin{equation}\label{5}
      \sigma_D = \{ \gamma_j \}_{j=1}^{\infty}, \ 0< \gamma_1 < \gamma_2 < \cdots
  \end{equation}

  Let $\sigma_S$ be the spectrum of the stokes operator
  (i.e. $-P \triangle $ where $P$ is the Leray projector)
  with homogenous Dirichlet boundary conditions (which we refer to classical Stokes operator),
  and we have
  \begin{equation}\label{6}
    \sigma_S = \{ \widetilde{\lambda}_j \}_{j=1}^{\infty},
      \quad 0< \widetilde{\lambda}_1 \leq \widetilde{\lambda}_2 \leq \cdots
  \end{equation}
  According to Theorem 1.1 in \cite{Kelliher}, we have $\gamma_k < \widetilde{\lambda}_k$
  for all positive integers.

  Back to the operator A of non-Newtonian fluids, we have
  \begin{equation}\label{7}
    \begin{split}
      \lambda_j e_j &= A e_j \\
                    &= P \triangle^2 e_j \\
                    &= -P \triangle (-P \triangle) e_j \\
                    &= \widetilde{\lambda}_j^2 e_j
    \end{split}
  \end{equation}
  Thus we can reorder the index and estimate the spectrum of A
  \begin{equation}\label{8}
    \lambda_{mn} = \widetilde{\lambda}_{mn}^2 \geq \gamma_{mn}^2 = (m^2+n^2)^2.
  \end{equation}
\end{proof}

  For technical reason we need the lemma:
  \begin{lemma}\label{Lemma: 2}
    The following inequality holds for all $\lambda>0$.
    \begin{equation}
       \int_0^{\lambda} e^{-2x}
          \left( \int_0^x  (e^y -1) y^{H-\frac{3}{2}}  dy  \right)^2 dx
       =C(H) < \infty.
  \end{equation}
  \end{lemma}

  \begin{proof}
   Calculate the integration and we have
   \begin{equation}
    \begin{split}
      & \int_0^{\lambda} e^{-2x}
          \left( \int_0^x  (e^y -1) y^{H-\frac{3}{2}}  dy  \right)^2 dx  \\
      \leq & \int_0^{\lambda} e^{-2x}
          \left( \int_0^1  (e^y -1) y^{H-\frac{3}{2}}  dy
                 + \int_1^{\max\{1,x\}}  (e^y -1) y^{H-\frac{3}{2}}  dy  \right)^2 dx .
    \end{split}
   \end{equation}
   Firstly we estimate the first term in the integral by parts of integration.
   \begin{equation}
    \begin{split}
        & \int_0^1  (e^y -1) y^{H-\frac{3}{2}}  dy \\
      = & \frac{1}{H-\frac{1}{2}}
          \left( (e^y-1)y^{H-\frac{1}{2}}\Big|_0^1  - \int_0^1 e^y y^{H-\frac{1}{2}} dy \right) \\
      = & \frac{1}{H-\frac{1}{2}}
          \left( (e-1)- \lim_{y \rightarrow 0} \frac{e^y-1}{y^{\frac{1}{2}-H}} \right)
                - \frac{1}{H-\frac{1}{2}} \cdot \frac{1}{H+\frac{1}{2}}
                   \int_0^1 e^y dy^{H+\frac{1}{2}}  \\
      = & \frac{1}{H-\frac{1}{2}} (e-1)
         -\frac{1}{H-\frac{1}{2}} \cdot \frac{1}{H+\frac{1}{2}}
          \left( e^y y^{H+\frac{1}{2}}\Big|_0^1  - \int_0^1 e^y y^{H+\frac{1}{2}} dy    \right)  \\
      = & C(H) < \infty .
    \end{split}
   \end{equation}

   Secondly we estimate the other term
   \begin{equation}
    \begin{split}
       & \int_0^{\lambda} e^{-2x}
          \left( \int_1^{\max\{1,x\}}  (e^y -1) y^{H-\frac{3}{2}}  dy  \right)^2 dx \\
     = & \int_1^{\lambda} e^{-2x}
          \left( \int_1^x  (e^y -1) \cdot y^{H-\frac{3}{2}}  dy  \right)^2 dx \\
     \leq & \int_1^{\lambda}
           \int_1^x         (e^y -1)^2 \cdot y^{2H-3} e^{-2x}(x-1) dy dx \\
     \leq & \int_1^{\lambda}
           \int_y^{\lambda} (e^y -1)^2 \cdot y^{2H-3} e^{-2x}(x-1) dx dy \\
     \leq & \int_1^{\lambda} (e^y -1)^2 \cdot y^{2H-3}
            \int_y^{\infty} e^{-2x}(x-1) dx dy \\
     \leq & \int_1^{\lambda} (e^y -1)^2 \cdot y^{2H-3}
            \left( \frac{1}{2}(y-1)e^{-2y} + \frac{1}{4} e^{-2y} \right) dy \\
     \leq & C \int_1^{\lambda} (1- e^{-y})^2 \cdot y^{2H-2}
             + (1-e^{-y})^2 \cdot y^{2H-3}  dy \\
     \leq & C \int_1^{\infty}  y^{2H-2} + y^{2H-3}  dy \\
     =    & C \left( \frac{1}{2H-1}+\frac{1}{2H-2} \right) = C(H) < \infty .
    \end{split}
   \end{equation}
   Finally we have
   \begin{equation}
    \begin{split}
      \leq & \int_0^{\lambda} e^{-2x}
          \left( \int_0^1  (e^y -1) y^{H-\frac{3}{2}}  dy
                 + \int_1^{\max\{1,x\}}  (e^y -1) y^{H-\frac{3}{2}}  dy  \right)^2 dx  \\
      \leq & \int_0^{\lambda} e^{-2x}
          \left( C(H)
                 + \int_1^{\max\{1,x\}}  (e^y -1) y^{H-\frac{3}{2}}  dy  \right)^2 dx \\
      \leq & 2\int_0^{\lambda} e^{-2x} C(H)^2 dx
             +2\int_0^{\lambda} e^{-2x}
                \left( \int_1^{\max\{1,x\}}  (e^y -1) y^{H-\frac{3}{2}}  dy  \right)^2 dx \\
      \leq & C(H)\int_0^{\infty} e^{-2x} dx   + C(H)
       < \infty .
    \end{split}
   \end{equation}
  \end{proof}

\begin{remark}
  There is a mistake in the proof of Lemma 2 \cite{Tindel Tudor Viens}.
  That is (in Page203), the quantity
  \begin{equation}
    K_A = \int_0^\infty \left( \int_0^x (e^y - 1)y^{A-1} dy \right)^2 dx
  \end{equation}
  is infinite.
  Therefore the theorems which use this lemma
  (such as Theorem 1.1 in \cite{Tindel Tudor Viens} and Theorem 2.1 in \cite{Wang Zeng Guo}) may not hold.
  However, the argument there can be modified by using our lemma
  and the expected result can be obtained.
\end{remark}

  The statement of existence theorem for stochastic convolution follows.
  \begin{theorem}
    If the Hurst parameter $H> \frac{1}{4}$,
    the stochastic convolution $\int_0^t S(t-s) dB^H(s) $ is well defined.
  \end{theorem}

  \begin{proof}
    The proof is base on Theorem 1 of \cite{Tindel Tudor Viens}.
    It is sufficient to estimate the mean square of the Wiener integral of \eqref{Equation: mild solu of OU}.
    \begin{equation}
        E \left| \int_0^t S(t-s) dB^H(s)  \right|_H^2
        = E \left| \sum_{n=1}^{\infty} \int_0^t S(t-s) e_n d\beta_n^H(s)  \right|_H^2
    \end{equation}
    Using \eqref{Equation: K* representation} and
    the representation \eqref{Equation: fBm represetation by Bm}, we have
    \begin{equation}
     \begin{split}
        & E \left| z(t) - S(t)u_0  \right|_H^2 \\
        \leq & 2 \sum_n \int_0^t | S(t-s)e_n|^2 K^2(t,s) ds
         + 2 \sum_n \int_0^t \left| \int_s^t (S(t-r)e_n -S(t-s)e_n)
           \frac{\partial K}{\partial r}(r,s) dr\right|^2 ds \\
        \triangleq & I_1+I_2
     \end{split}
    \end{equation}
    By \cite{Decreusefond Ustunel} Th3.2, we have
    \begin{equation}
        K(t,s) \leq C(H) (t-s)^{H-\frac{1}{2}} s^{H-\frac{1}{2}}.
    \end{equation}
    Then,
    \begin{equation}
     \begin{split}
        I_1 & \leq C(H) \sum_n \int_0^t |e^{-(t-s)\lambda_n} e_n|^2 (t-s)^{2H-1} s^{2H-1} ds \\
            & = C(H) \sum_n (2\lambda_n)^{-2H}
              \int_0^{2\lambda_n t} e^{-v} v^{2H-1} (t-\frac{v}{2\lambda_n})^{2H-1} dv
     \end{split}
    \end{equation}
    where $C(H)$ depends only on $H$ and we use the change of variable $t-s=\frac{v}{2\lambda_n}$.
    Since
    \begin{equation}
     \begin{split}
       & \int_0^{2\lambda_n t} e^{-v} v^{2H-1} (t- \frac{v}{2\lambda_n})^{2H-1} dv \\
       \leq &  \int_0^{\lambda_n t} e^{-v} v^{2H-1} (\frac{t}{2})^{2H-1} dv
             +  \int_{\lambda_n t}^{2\lambda_n t}
                         e^{-v} {2\lambda_n t}^{2H-1} (t- \frac{v}{2\lambda_n})^{2H-1} dv \\
       \leq & C(t,H) \int_0^{\infty} e^{-v} v^{2H-1} dv
         +(2\lambda_n t)^{2H-1} \int_0^{\lambda_n t}
                  e^{-(2\lambda_nt-v')} (\frac{v'}{2\lambda_n})^{2H-1} dv' \\
       \leq & C(t,H) \Gamma(2H) + t^{2H-1} \int_0^{\lambda_n t}
                  e^{-2\lambda_n t} e^{v'} v'^{2H-1} dv'  \\
       \leq & C(t,H) + C(t,H) e^{-\lambda_n t} \int_0^{\lambda_n t}v'^{2H-1} dv' \\
       \leq & C(t,H) \cdot e^{-\lambda_n t}(\lambda_n t)^{2H} \\
       \leq & C(t,H),
     \end{split}
    \end{equation}
    we have
    \begin{equation}
        \begin{split}
          I_1
          & \leq C(H) \cdot \sum_n \lambda_n^{-2H} \cdot C(t,H) \\
          & \leq C(t,H) \cdot \sum_{i,j=1}^\infty \frac{1}{(i^2 + j^2)^{4H}} \\
          & \leq C(t,H) \cdot 2 \beta_D(4H) \cdot \zeta(4H) < \infty
        \end{split}
    \end{equation}
    where $\beta_D(s)$ is the Dirichlet beta function and
    $\zeta(s)$ is the Riemann zeta function (for definition see \cite{Borwein Borwein}).

    For the second sum, we have
    \begin{equation}
      \begin{split}
        I_2
        & = 2 \sum_n \int_0^t
          \left| \int_s^t (S(t-r)e_n-S(t-s)e_n) \frac{\partial K}{\partial r}(r,s) dr \right|_H^2 ds \\
        & \leq 2 \sum_n \int_0^t
          \left| \int_s^t (e^{-(t-r)\lambda_n} - e^{-(t-s)\lambda_n})
                  \frac{\partial K}{\partial r}(r,s) dr \right|^2 ds
      \end{split}
    \end{equation}
    By the fact that $\frac{\partial K}{\partial r}(r,s) \leq 0$ for every $r,s \in [0,T]$ and
    \begin{equation}
        \left| \frac{\partial K}{\partial r}(r,s) \right| \leq C(H) \cdot (r-s)^{H-\frac{3}{2}},
    \end{equation}
    we have
    \begin{equation}
      \begin{split}
        I_2
        & \leq C(H) \sum_n \int_0^t
          \left( \int_s^t (e^{-(t-r)\lambda_n} - e^{-(t-s)\lambda_n})
                  (r-s)^{H-\frac{3}{2}} dr \right)^2 ds \\
        & \leq C(H) \sum_n \int_0^t
          \left( \int_0^{t-s} (e^{-(t-s-v)\lambda_n} - e^{-(t-s)\lambda_n})
                  v^{H-\frac{3}{2}} dv \right)^2 ds       \\
        & \leq C(H) \sum_n \int_0^t
         \left( \int_0^{u} (e^{-(u-v)\lambda_n} - e^{-u\lambda_n})
                  v^{H-\frac{3}{2}} dv \right)^2 du       \\
      \end{split}
    \end{equation}
  By the change of variables $v=\frac{y}{\lambda_n}, \ u= \frac{x}{\lambda_n}$, we have
  \begin{equation}
    \begin{split}
      I_2
      & \leq C(H) \sum_n \int_0^{\lambda_n t} e^{-2x}
        \left( \int_0^x  (e^y -1) (\frac{y}{\lambda_n})^{H-\frac{3}{2}} \lambda_n^{-1} dy  \right)^2 \lambda_n^{-1} dx \\
      & \leq C(H) \sum_n \lambda_n^{-2H} \int_0^{\lambda_n t} e^{-2x}
        \left( \int_0^x  (e^y -1) y^{H-\frac{3}{2}}  dy  \right)^2 dx \\
      & \leq C(H) \sum_n \lambda_n^{-2H} C(H) \qquad (\text{by lemma \ref{Lemma: 2}}) \\
      & \leq C(H) \cdot 2 \beta_D(4H) \cdot \zeta(4H) < \infty \quad \text{ as } 4H>1.
    \end{split}
  \end{equation}
  \end{proof}

In the following we investigate the regularity of solution.
First we show that $B^H \in C([0,\infty);V')$.

Since
\begin{equation}
  \begin{split}
        \sum_{n=1}^{\infty} \left| A^{-\frac{1}{2}} e_n \right|
        &= \sum_{n=1}^{\infty} \lambda_n^{-1} \\
        &= 2 \beta_D(2) \cdot \zeta(2) < \infty,
  \end{split}
\end{equation}
$A^{-\frac{1}{2}} : H \rightarrow V'$ is a Hilbert-Schmidt operator from $H$ to $V'$.
Thus the genuine cylindrical fBm $B^H$ can be viewed as a $V'$-valued fBm
with incremental covariance operator $A^{-\frac{1}{2}}$.
And we have that $B^H$ has continuous trajectories on $V'$.

Let
\begin{align}
     z(t):   &= \int_0^t S(t-s)dB^H (s), \\
     Y(t):   &= \int_0^t S(t-s) B^H(s)ds,
\end{align}
and we have

\begin{lemma}
  For $H>\frac{1}{4}$, $Y(\cdot)$ belongs to $C^1([0,\infty);V)$ P-a.s. and
  \begin{equation}
    \begin{split}
      \frac{d}{dt}Y(t) &= \frac{d}{dt} \int_0^t S(t-s) B^H(s)ds \\
      & = B^H(t) + A\int_0^t S(t-s) B^H(s)ds \\
      &= z(t).
    \end{split}
  \end{equation}

\end{lemma}

Since the proof is a modification (consider the base space as $V'$) of Lemma 5.13 in \cite{DaPrato Zabczyk} and
Proposition 3.1 in \cite{Maslowski Schmalfus},
we only sketch the main ideas.

Step I:

When $A_0: V' \rightarrow V'$ is a bounded linear operator, we have
\begin{equation}
    z(t) = \int_0^t A_0 z(s)ds +B^H(t), \quad t \geq 0, \text{P-a.s.}
\end{equation}

Step II:

Let $A_n = n(nI-A)^{-1} A$, $n \in \mathbb{N}$ (n large enough)
denote the Yosida approximations of $A$.
Denote by $S_n(\cdot)$ the semigroup generated by $A_n$ in $V'$ and
set $z_n(t)=\int_0^t S_n(t-s)dB^H(s)$.
By using the so-called factorization method, we have for $q>2$
\begin{equation}
    E \sup_{t \in [0,T]} |z_n(t)-z(t)|^2_{V'} \rightarrow 0,
    \quad \text{as } n \rightarrow \infty.
\end{equation}

Step III:

Since $z_n(\cdot)$ is a strong solution to the stochastic differential equation in $V'$
\begin{equation}
    dz_n(t)= A_nz_n(t)dt + dB^H(t), \quad z_n(0)=0.
\end{equation}
We have
\begin{equation}
    z_n(t) = \int0^t A_n z_n(s)ds + B^H(t).
\end{equation}
Setting $Y_n(t)=\int_0^t z_n(s)ds$, we have $Y_n(\cdot)$ is the solution to the initial value problem
\begin{equation}
    \frac{d}{dt}Y_n(t) = A_n Y_n(t) + B^H(t), \quad Y_n(0)=0.
\end{equation}
Thus we have
\begin{equation}
    Y_n(t)= \int_0^t S_n(t-s) B^H(s)ds,
\end{equation}
and so P-a.s.
\begin{equation}
    z_N(t) = A_n Y_n(t) + B^H(t).
\end{equation}
clearly
\begin{equation}
    \lim_{n\rightarrow \infty} Y_n(t) = Y(t) = \int_0^t S(t-s) B^H(s)ds.
\end{equation}
By step I and step II, we have
\begin{equation}
    \lim_{n \rightarrow  \infty} A_n Y_n(t)
    =\lim_{n \rightarrow  \infty} A(I-\frac{1}{n}A)^{-1} Y_n(t)
    = z(t) - B^H(t).
\end{equation}
Since the operator $A$ is closed, we conclude that $Y(t) \in V$ and
$AY(t)=z(t)-B^H(t)$ P-a.s..

The following proposition is an immediate consequence of above lemma.

\begin{proposition}\label{Lemma: continuous modification for fBm}
  If $H \in (\frac{1}{4},\frac{1}{2})$.
  Then for all $u_0 \in V$
  the process $Z(t,x)=S(t)u_0+z(t)$ has an V-continuous modification.
  Moreover,
  \begin{equation}
    z(t)= A \int_0^t S(t-s) B^H(s) ds + B^H(t), \quad t\geq 0
  \end{equation}
  holds with probability one.
\end{proposition}

\section{Solution of stochastic non-Newtonian fluids}
We interpret equation \eqref{Problem: differential form} as an integral equation and
define the solution as follow.
\begin{definition}
  The solution to equation \eqref{Problem: differential form} is defined as a function
  $u \in C([0,T];H)\cap L^2(0,T;V)$  s.t. the following integral equation holds P-a.s..
  \begin{equation}\label{Equation: integral form}
    u(t)=S(t)u_0 -\int_0^T S(t-s)B(u(s))ds -\int_0^T S(t-s) N(u(s))ds
    + \int_0^T S(t-s)dB^H(s)
  \end{equation}
  where the first and second integral are Bochner integral,
  the last integral is a stochastic integral defined in previous section.
\end{definition}

We seek the solution of 2D stochastic Bellout-Bloom-Ne\v{c}as fluids by the fixed point in space
$X=C([0,T];H)\cap L^2(0,T;V)$.
Firstly we will prove the local existence and uniqueness results.

For $ u \in X$, let
\begin{align}
  J_1(u): &= - \int_0^{\cdot} S(\cdot-s)B(u(s))ds , \\
  J_2(u): &= - \int_0^{\cdot} S(\cdot-s)N(u(s))ds .
\end{align}
We have following estimates.

\begin{lemma}\label{Lemma: estimate J1}
  \cite{Li Huang}
  $J_1 : X \rightarrow X $ and
  for all $ u,v \in X$, we have
    \begin{align}
      |J_1(u)|^2_X & \leq c_1 |u|_X^4 ,      \label{Estimate: J1 1}\\
      \begin{split}
        |J_1(u)-J_1(v)|^2_X & \leq c_2 \left( |u|^2_{C([0,T];H)}\cdot |u|^2_{L^2(0,T);V}
       + |v|_{C([0,T];H)}^2 \cdot |v|^2_{L^2(0,T);V} \right)^{\frac{1}{2}}  \label{Estimate: J1 2}\\
       & \quad \cdot |u-v|^2_X.
      \end{split}
    \end{align}
\end{lemma}

\begin{lemma}\label{Lemma: estimate J2}
  \cite{Li Huang}
  $J_2 : X \rightarrow X $ and
  for all $ u,v \in X$, we have
    \begin{align}
      |J_2(u)|^2_X        & \leq  c_3 |u|^2_{L^2(0,T;V)}   ,   \label{Estimate: J2 1}\\
      |J_2(u)-J_2(v)|^2_X & \leq  c_4 T |u-v|^2_X          .   \label{Estimate: J2 2}
    \end{align}
\end{lemma}

%

We apply the following version of the contraction mapping theorem.

\begin{lemma}\label{Lemma: fix point}
  (\cite{DaPrato Zabczyk} Lemma 15.2.6)
  Let $F$ be a transformation from a Banach space $E$ into $E$,
  $\phi \in E$ and $M>0$ a positive number.
  If $F(0)=0$, $| \phi |_E \leq \frac{1}{2}M$ and
  \begin{equation}
    |F(u)-F(v)|_E \leq \frac{1}{2} |u-v|_E \quad \forall u,v \in B_E(M),
  \end{equation}
  then the equation
  \begin{equation}
    u = \phi +F(u)
  \end{equation}
  has a unique solution $u \in E$ satisfying $u \in B_E(M)$.
\end{lemma}

We now prove the main result of this paper.
\begin{theorem}\label{Theorem: existence of solu}
  (local existence and uniqueness of solution) For all $u_0 \in H$, there exists $T_0>0$ s.t.
  equation \eqref{Problem: differential form} P-a.s. admits a unique solution
  $u \in C([0,T_0];H) \cap L^2(0,T_0;V)$ in the sense of \eqref{Equation: integral form}.
\end{theorem}

\begin{proof}
  Fix $\omega \in \Omega$. Let
  \begin{equation}
    \phi (t)= S(t)u_0 + z(t).
  \end{equation}
  By the properties of semigroup $S(\cdot)$ and
  Lemma \ref{Lemma: continuous modification for fBm},
  $S(\cdot)u_0,z \in C([0,T];V) \subset X$.
  Then we have
  \begin{equation}
    |\phi |_X \leq |S(\cdot)u_0|_X+|z|_X \leq 2|u_0|+|z|_X.
  \end{equation}
  Let $M(\omega)=2(2|u_0|+|z(\omega)|_X)$.

  Construct the mapping $\mathcal{F}=J_1+J_2$,
  then for all $u,v \in X$, we have
  \begin{equation}
   \begin{split}
    & \quad |\mathcal{F}(u)-\mathcal{F}(v)|_X  \\
    & \leq |J_1(u)-J_1(v)|_X + |J_2(u)-J_2(v)|_X \\
    & \leq  c_2^{\frac{1}{2}}  \left( |u|^2_{C([0,T];H)}\cdot |u|^2_{L^2(0,T);V}
       + |v|_{C([0,T];H)}^2 \cdot |v|^2_{L^2(0,T);V} \right)^{\frac{1}{4}}
      \cdot |u-v|_X  \\
    & \quad +  (c_4  T)^{\frac{1}{2}} |u-v|_X \qquad
       \text {(by Lemma \ref{Lemma: estimate J1} and \ref{Lemma: estimate J2})}   \\
    & \leq  (c_2  M)^{\frac{1}{2}}  \left( |u|^2_{L^2(0,T);V}
       +   |v|^2_{L^2(0,T);V} \right)^{\frac{1}{4}}  |u-v|_X
       +   (c_4 T)^{\frac{1}{2}} |u-v|_X .
   \end{split}
  \end{equation}
  Due to the absolute continuity property of Bochner integral,
  we can choose $ \tau \in (0,1]$ s.t.
  \begin{equation}
    \left(|u|^2_{L^2(0,\tau);V} + |v|^2_{L^2(0,\tau);V} \right)^{\frac{1}{4}}
    \leq (2Mc_2)^{-\frac{1}{2}}.
  \end{equation}
  Let $T_0 = \min \{\tau, 1, \frac{1}{16c_4}  \}$ and $X_{T_0}:=C([0,T_0];H) \cap L^2(0,T_0;V)$.
  We have
  \begin{equation}
     |\mathcal{F}(u)-\mathcal{F}(v)|_{X_{T_0}}
     \leq (\frac{1}{4}+\frac{1}{4})|u-v|_{X_{T_0}} = \frac{1}{2}|u-v|_{X_{T_0}}.
  \end{equation}
  Applying the modified fixed point lemma \ref{Lemma: fix point},
  equation
  \begin{equation}
    u= \phi + \mathcal{F}(u) \equiv S(\cdot)u_0 + z + J_1(u)+J_2(u)
  \end{equation}
  has a unique solution $u$  in $C([0,T_0];H) \cap L^2(0,T_0;V)$
  and the solution satisfies $|u|_{X_{T_0}} \leq M$.
\end{proof}

Secondly we give a priori estimates and  obtain the global existence.
Denote by $u$, the local solution of \eqref{Equation: integral form} over $[0,T_0]$.
Let $v(t)=u(t)-z(t)$.
Then $v(t)$ is the mild solution of equation
\begin{equation}\label{Equation: integral form for v}
    v(t)=S(t)u_0 - \int_0^t S(t-s)B(v(s)+z(s))ds - \int_0^t S(t-s)N(v(s)+z(s))ds.
\end{equation}

According to  \cite{Chueshov} 2.1.20,
$v(t)$ is the weak solution of the following differential equation with random parameters:
\begin{equation}\label{Equation: differential form for v}
   \left\{\begin{split}
      & \frac{v(t)}{dt}+Av(t)+B(v(t)+z(t))+N(v(t)+z(t))=0, \\
      & v(0)=u_0
      \end{split}
  \right.
\end{equation}

Inspired by \cite{DaPrato Zabczyk} Chapter 15.3,
we take advantage of the relation between weak and mild solution,
give a priori estimate which ensures the global existence of solution.

\begin{proposition}\label{Proposition: Extention of solu}
  \cite{Li Huang}
  Assume that $v$ is the solution of \eqref{Equation: integral form for v} on the interval.
  Then we have
\begin{align}
  \begin{split}\label{Equation: estimate v in CH}
    \sup_{t \in [0,T]} |v(t)|^2
    \leq e^{c_5 \int_0^T |z(s)|_{H_0^1}^2 ds} |u_0|^2
    + \int_0^T e^{c_5 \int_s^T |z(r)|_{H_0^1}^2 dr} g_1(s) ds,
  \end{split} \\
  \begin{split}\label{Equation: estimate v in LV}
    \int_0^T |v(t)|_V^2 dt
    \leq c_6 |u_0|^2
    + c_5 c_6  \sup_{t \in [0,T]} |v(t)|^2 \int_0^T |z(s)|_{H_0^1}^2 ds
    + c_6 \int_0^T g_1(s) ds
  \end{split}
\end{align}
  where $c_5$ and $c_6$ are positive constants depending on $\lambda_1$ and $\mathcal{O}$,
  $g_1$ is an integrable function depending on $z$.
\end{proposition}

Since $z \in C([0,T];V)$, the following theorem is an immediate consequence of theorem \ref{Theorem: existence of solu}
and proposition \ref{Proposition: Extention of solu}.
\begin{theorem}
   For all $ T >0$ and $u_0 \in H$,
   the equation \eqref{Problem: differential form} P-a.s.
   has a unique solution $u \in C([0,T];H) \cap L^2(0,T;V)$
   in the sense of \eqref{Equation: integral form}.
\end{theorem}

\section{Random attractor}
First we introduce the fractional Ornstein-Uhlenback process (fractional O-U process) to
construct the random dynamical system.

Consider the linear stochastic evolution equation

\begin{equation}
    dz(t)=Az(t)+dB^H(t), \quad t \in \mathbb{R}.
\end{equation}
By the similar argument in \cite{Maslowski Schmalfus} chapter 3,
we can construct a unique stationary solution
$Z(t)$ (the so-called fractional O-U process) such that
  \begin{align}
    Z(t) = Z(\theta_t \omega) = \int_{-\infty}^{t} S(t-r)dB^H (r)  \\
    Z(\omega)
     = \lim_{n \rightarrow \infty} A \int_{-n}^0 S(-r)dB^H (r)
  \end{align}

Next we prove
    \begin{equation}
        \lim_{n\rightarrow \pm \infty} \frac{1}{n}
         \int_0^n \| Z (\theta_t \omega) \|_{H_0^1}^2 dt
        \leq 2c_0 \beta_D(4) \cdot \zeta(4).
    \end{equation}
    where $c_0$ is a constant depending on the space injection.
    In order to convert this integration with respect to time variable into
    a integration with respect to sample space,
    we may use ergodic theory.
    Consider the real-valued continuous function $|Z(\theta_{\cdot} \omega)|_{H_0^1}^2$, we have
    \begin{equation}
      \begin{split}
    \mathbb{E} \left| Z(\omega)  \right|_{H_0^1}^2
      & = \mathbb{E} \left| \int_{-\infty}^0 S(-r)dB^H(r) \right|_{H_0^1}^2 \\
      & = \mathbb{E} \left| \lim_{t \rightarrow \infty}
         \sum_{n=1}^{\infty}
         \int_{-t}^0 S(-r) e_n d \beta_n^H (r) \right|_{H_0^1}^2       \\
      & = \sum_{n=1}^{\infty}   \lim_{t \rightarrow \infty}
         \mathbb{E} \left| \int_{-t}^0 S(-r) e_n d \beta_n^H (r) \right|_{H_0^1}^2 \\
      & = \sum_{n=1}^{\infty}
        \lim_{t \rightarrow \infty}
        \left|K^*_H(S(- \cdot) e_n) \right|_{L^2(-t,0;H_0^1)}^2  \\
      & = \sum_{n=1}^{\infty}
        \lim_{t \rightarrow \infty}
        \left|S(- \cdot) e_n \right|_{I_T^{\frac{1}{2}-H} (L^2(-t,0;H_0^1))}^2
          \quad  \text{ (Since } \mathcal{H} = I_T^{\frac{1}{2}-H} (L^2(-t,0;H_0^1))) \\
      & \leq \sum_{n=1}^{\infty}
        \lim_{t \rightarrow \infty}
        c_0 \left|S(- \cdot) e_n \right|_{L^2(-t,0;H_0^1)}^2 \quad
         \text{ (by Hardy-Littlewood theorem)} \\
      & =\sum_{n=1}^{\infty}
        \lim_{t \rightarrow \infty} c_0
        \int_0^t |A^{\frac{1}{4}} e^{-tA} e_n |^2 dt   \\
      & =\sum_{n=1}^{\infty}
        \lim_{t \rightarrow \infty} c_0
        \int_0^t | \lambda_n^{\frac{1}{4}} e^{-t \lambda_n} e_n |^2  dt   \\
      & =c_0 \sum_{n=1}^{\infty}
        \lim_{t \rightarrow \infty}
        \lambda_n^{-\frac{1}{2}} (1- e^{-t \lambda_n}  )     \\
      & = 2c_0 \beta_D(4) \cdot \zeta(4)
      \end{split}
    \end{equation}
    Hence, we have $|Z(\theta_{\cdot} \omega)|_{H_0^1}^2 \in L^1(\Omega, P)$.
    Since $(\Omega, \mathcal{F}, \{ \theta(t) \}_{t\in \mathbb{R}} )$
    is the metric dynamical system, we can use the Birkhoff-Chintchin Ergodic Theorem to obtain
    \begin{equation}\label{Equation: Ergodic}
        \lim_{n\rightarrow \pm \infty} \frac{1}{n}
         \int_0^n \| Z (\theta_t \omega) \|_{H_0^1}^2 dt
        = \mathbb{E} \| Z (\omega) \|_{H_0^1}^2
        \leq 2c_0 \beta_D(4) \cdot \zeta(4).
    \end{equation}

According to Theorem \ref{Theorem: existence of solu}, $\forall t_0 \in \mathbb{R}$,
$u(t,\omega;t_0,u_0)$ is the unique solution of the equation
\begin{equation}\label{Equation: integral form u t t0}
    u(t;t_0)
    = S(t-t_0)u_0
     - \int_{t_0}^t S(t-s) B(u(s))ds
     - \int_{t_0}^t S(t-s) N(u(s))ds
     + \int_{t_0}^t S(t-s) dB^H(s).
\end{equation}

In this section, let $u(t,\omega;t_0)=v(t,\omega;t_0) + Z(t,\omega)$, we have
\begin{equation}
  \begin{split}
      & v_2(t)+ \int_{-\infty}^t S(t-s) dB^H(s) \\
    = & S(t)u_0
      -\int_{t_0}^t S(t-s) B( v_2(s) + Z(s) ) ds
      -\int_{t_0}^t S(t-s) N( v_2(s) + Z(s) ) ds
      + \int_{t_0}^t S(t-s) dB^H(s) .
  \end{split}
\end{equation}
Since
\begin{equation}
    \int_{-\infty}^{t_0} S(t-s)dB^H(s)
    = S(t-t_0) Z (\theta_{t_0}\omega),
\end{equation}
$v(t,\omega;t_0,u_0-Z(\theta_{t_0}\omega))$ is the unique solution of the integral equation
\begin{equation}\label{Equation: integral form for v infty}
    v(t)
    =S(t) (u_0-Z(\theta_{t_0}\omega))
     -\int_{t_0}^t S(t-s) B( v(s) + Z(s) ) ds
     -\int_{t_0}^t S(t-s) N( v(s) + Z(s) ) ds.
\end{equation}
According to \cite{Chueshov} 2.1.20,
$v$ is the weak solution of the following differential equation
\begin{align}
    & \frac{dv}{dt} + A(v+Z) + B(v+Z) = 0 \label{Equation: differ form for v infinity}\\
    & v(t_0) = u_0 - Z(\theta_{t_0}\omega) \label{Equation: boundary for v infinity}
\end{align}

We can now define an continuous mapping by setting
\begin{equation}
    \phi(t,\omega,u_0)= v(t,\omega; 0,u_0-Z(\omega)) + Z(\theta_t \omega),
    \quad \forall (t,\omega,u_0) \in \mathbb{R} \times \Omega \times H.
\end{equation}
The measurability follows from the continuity dependence of solution with respect to initial value.
the cocycle property follows from the uniqueness of solution.
Thus, $\phi$ is a RDS associated with \eqref{Equation: integral form}.

In the rest of this section, we will compute some estimates in spaces $H$ and $V$.
Then we use these estimates and compactness of the embedding $V \hookrightarrow H$
to obtain the existence of a compact random attractor.
Assume that $C_1$ is the constant which satisfies the inequality of trilinear form $b$
\begin{equation}
    b(u,v,w) \leq C_1 |u|^{1/2} \cdot |u|^{1/2}_{H_0^1} \cdot |v|_{H_0^1}
       \cdot |w|^{1/2} \cdot |w|^{1/2}_{H_0^1}
\end{equation}
as in \cite{RogerTemam}.

\begin{lemma}\label{Lemma: absorb in H}
  If $c_0 C_1^2 < \frac{1}{\beta_D(4)  \zeta(4)} $,
  then there exist random radii  $\rho_H(\omega)>0$ and $\rho_1(\omega)$ such that
  for all $ M>0$ there exists $t_2(\omega) < -1$,
  such that whenever $t_0 < t_2$ and $|u_0| <M$, we have
  \begin{align}
    | v(t, \omega; t_0,u_0 - Z(\theta_{t_0}\omega))|^2 &\leq \rho_H(\omega),
        \quad \forall t \in [-1,0] \\
    | u(t, \omega; t_0,u_0)|^2 &\leq \rho_H(\omega) ,
         \quad \forall t \in [-1,0]\\
    \int_{-1}^0 |v(t)|_V^2 dt &\leq \rho_1(\omega)  \\
    \int_{-1}^0 |v(t)+Z(t)|_V^2 dt &\leq \rho_1(\omega)
  \end{align}
\end{lemma}

\begin{proof}
  The proof is similar to \cite{Li Huang} Proposition 3.7.
  Multiple \eqref{Equation: differential form for v} by $v(t)$
and then integrate over $\mathcal{O}$. We have
\begin{equation}
  \begin{split}
    \frac{1}{2} \frac{d|v(t)|^2}{dt} +|v(t)|^2_V
    & = - b\left(v(t)+Z(t),v(t)+Z(t),v(t)\right)  - <N(v(t)+Z(t)),v(t)> \\
    &\leq |b \left( v(t)+Z(t),Z(t),v(t)+Z(t)\right)| - <N(Z(t)),v(t)>.
  \end{split}
\end{equation}
The above inequality take advantage of $<N(v),v> \ \geq 0$ (see \cite{Bloom Hao existence})
and \eqref{Equation: orthogonal for b}.

In the sequel we omit the time variable $t$.
Firstly we estimate trilinear form $b$.
\begin{equation}
    \begin{split}
      & \quad b\left(v+Z,Z,v+Z\right) \\
      & \leq C_1 |v+Z| \cdot
        |Z|_{H_0^1} \cdot
        |v+Z|_{H_0^1}  \\
      & \leq \frac{C_1}{2C_2} |Z|^2_{H_0^1} \cdot |v+Z|^2
        + \frac{C_1 C_2}{2} |v+Z|^2_{H_0^1} \\
      & \leq \frac{C_1}{C_2} |Z|_{H_0^1}^2 |v|^2
        + C_1 C_2 |v|_{H_0^1}^2
        + \frac{C_1}{C_2} |Z|^2 |Z|_{H_0^1}^2
        + C_1 C_2 |Z|_{H_0^1}^2,
    \end{split}
\end{equation}
where $C_2$ is a positive constant which will be specified later.
Secondly we estimate nonlinear term $N$.
For all $r_1 >0$, we have
\begin{equation}
  \begin{split}
           -<N(Z),v> \
     \leq & \mu_0 \epsilon^{- \alpha /2} |Z|_{H_0^1} |v|_{H_0^1} \\
     \leq & r_1 |v|_{H_0^1}^2 + \frac{\mu_0^2 }{4r_1 \epsilon^{\alpha}} |Z|_{H_0^1}^2.
  \end{split}
\end{equation}
Comprehensively,
\begin{equation}
  \begin{split}
      & \frac{1}{2} \frac{d}{dt}|v|^2
      +\frac{\lambda_1}{2}|v|^2
      +\frac{1}{2} |v|^2_V \\
    \leq & \frac{C_1 }{C_2} |Z|_{H_0^1}^2 |v|^2
      + (C_1 C_2 +r_1)|v|_{H_0^1}^2
      + \frac{C_1}{C_2} |Z|^2 |Z|_{H_0^1}^2
      + C_1 C_2 |Z|_{H_0^1}^2
      +  \frac{\mu_0^2 }{4r_1 \epsilon^{\alpha}} |Z|_{H_0^1}^2,
  \end{split}
\end{equation}
where $\lambda_1>4$ is the first eigenvalue of operator $A$.
  Let $g_2
      =\frac{C_1}{C_2} |Z|^2 |Z|_{H_0^1}^2
      + C_1 C_2 |Z|_{H_0^1}^2
      +  \frac{\mu_0^2 }{4r_1 \epsilon^{\alpha}} |Z|_{H_0^1}^2$.
Then we have
  \begin{equation}\label{Equation: estimate v infty diff form}
    \frac{d}{dt} |v|^2
      + \left( \frac{1}{2}-\frac{C_1 C_2 +r_1 }{\lambda_1^{\frac{1}{2}}} \right) |v|_V^2
      + \left( \frac{\lambda_1}{2} - \frac{C_1 |Z|_{H_0^1}^2}{C_2}  \right) |v|^2
    \leq g_2 .
  \end{equation}
By assumption we can choose
$C_2 \in \left( c_0 C_1 \beta_D(4) \zeta(4)  ,   \frac{1}{C_1} \right)$
and $r_1$ small enough and we have
\begin{equation}\label{Equation: estimate v infty diff simp form }
    \frac{d}{dt} |v|^2
      + \left( \frac{\lambda_1}{2}  - \frac{C_1 |Z|_{H_0^1}^2}{C_2}  \right) |v|^2
    \leq g_2 .
\end{equation}

By Gronwall inequality, when $t\in [-1,0]$ and $t_0<-1$, we have
\begin{equation}
  \begin{split}
      & |v(t)|^2  \\
     \leq & |v(t_0)|^2 e^{ -\int_{t_0}^{t} \left( \frac{\lambda_1}{2} - \frac{C_1 |Z(s)|_{H_0^1}^2}{C_2} \right) ds }
       + \int_{t_0}^{t} g_2(s_1) e^{ -\int_{s_1}^{t}
       \left( \frac{\lambda_1}{2} - \frac{C_1 |Z(s_2)|_{H_0^1}^2}{C_2} \right)ds_2 }ds_1 \\
     \leq & |v(t_0)|^2 e^{ -\int_{t_0}^{0} \left( \frac{\lambda_1}{2} - \frac{C_1 |Z(s)|_{H_0^1}^2}{C_2} \right) ds }
       + \int_{t_0}^{0} g_2(s_1) e^{ -\int_{s_1}^{0}
       \left( \frac{\lambda_1}{2} - \frac{C_1 |Z(s_2)|_{H_0^1}^2}{C_2} \right)ds_2 }ds_1.
  \end{split}
\end{equation}
Due to the ergodic property of fractional O-U process \eqref{Equation: Ergodic}, we have
\begin{equation}
    \lim_{t_0 \rightarrow -\infty} \frac{1}{-t_0}
    \int_{t_0}^0 |Z(s)|^2_{H_0^1} ds
    = \mathbb{E} |Z(\omega)|_{H_0^1} .
\end{equation}
Choose $r_2$ small enough such that
\begin{equation}
    \frac{C_1}{C_2} \mathbb{E} |Z(\omega)|_{H_0^1}
    \leq \frac{C_1}{C_2} 2c_0
    \beta_D(4) \cdot \zeta(4)
    < 2 - r_2
    \leq \frac{\lambda_1}{2} -r_2.
\end{equation}
Then there exists $t_1(\omega) < -1$, such that when $t_0 < t_1$ we have
\begin{equation}
    |v(t)|^2
    \leq e^{(1+t_0)r_2} |u_0|^2
    + \int_{t_0}^{0} e^{(1+t_0)r_2} g_2(s) ds,
         \quad \forall t \in [-1,0].
\end{equation}

By Lemma 2.6 of \cite{Maslowski Schmalfus},
$g_2$ has at most polynomial growth as $t_0 \rightarrow -\infty$ for P-a.s.
$\omega \in \Omega$.
Thus, we have
\begin{equation}
    \int_{t_0}^{0} g_2(s) e^{(1+s)r_2} ds
    \leq \int_{-\infty}^{0} g_2(s) e^{(1+s)r_2} ds
    \leq \infty, \quad \text{P-a.s.}
\end{equation}
Let $\rho_H = 4\int_{-\infty}^{0} g_2(s) e^{(1+s)r_2} ds + 2 \sup_{t \in [-1,0]} |Z(t)|^2$
and there exists $t_2(\omega) < t_1(\omega) < -1$ such that for all $|u_0| \leq M$
\begin{align}
    |v(-1,\omega ; t_0,u_0-Z(\theta_{t_0}\omega))|^2
      &\leq 2\int_{-\infty}^{0} g_2(s) e^{(1+s)r_2} ds \\
    \begin{split}
    |u(-1,\omega ; t_0,u_0)|^2
      &\leq 2 |v(-1,\omega ; t_0,u_0-Z(\theta_{t_0}\omega))|^2
      + 2 \sup_{t \in [-1,0]} |Z(t)|^2 \\
      &\leq \rho_H(\omega), \quad  \forall t_0 < t_2, t \in [-1,0].
    \end{split}
\end{align}

In the following we consider a bound of $\int_{-1}^0 |v(t)|_V^2 dt$.
Integrating \eqref{Equation: estimate v infty diff form} over $[-1,0]$ we have
\begin{equation}
    |v(0)|^2 - |v(-1)|^2 + c_6^{-1} \int_{-1}^0 |v(t)|_V^2 dt
    \leq \int_{-1}^{0} g_2(t)dt + \int_{-1}^0 ( \frac{C_1}{C_2} |Z(t)|_1^2 )|v(t)|^2 dt.
\end{equation}
When $t_0<t_2$ we have
\begin{equation}
      \int_{-1}^0 |v(t)|_V^2 dt
     \leq c_6 (
      \int_{-1}^{0} g_2(t)dt
      +  \frac{C_1 \rho_H}{C_2} \int_{-1}^0 |Z(t)|_1^2  dt
      + |v(-1)|^2)
     \triangleq C(\omega).
\end{equation}
Similarly,
\begin{equation}
      \int_{-1}^0 |v(t)+Z(t)|_V^2 dt
     \leq 2c_6 (
      \int_{-1}^{0} g_2(t)dt
      +  \frac{C_1 \rho_H}{C_2} \int_{-1}^0 |Z(t)|_1^2  dt
      + 2 \int_{-1}^0 |Z(t)|_V^2 dt)
     \triangleq \widetilde{C}(\omega) .
\end{equation}
Let $\rho_1(\omega) = \max \{C(\omega),\widetilde{C}(\omega)\}$ and the proof is complete.
\end{proof}

By the same argument as in \cite{Li Huang} Lemma 4.3, we have the following lemma.
\begin{lemma}\label{Lemma: absorb in V}
  Under the assumption of Lemma \ref{Lemma: absorb in H},
  there exists a random radius $\rho_V(\omega)$ such that for all $M>0$
  and $|u_0| < M $, there exists $t_2(\omega)<-1$ such that P-a.s.
  \begin{align}
     |v(t,\omega;t_0,u_0 - Z(\theta_{t_0}\omega))|_1^2
        & \leq \rho_V(\omega), \\
     |u(t,\omega;t_0,u_0 )|_1^2
        & \leq \rho_V(\omega), \quad \forall t_0 < t_2, t \in [-\frac{1}{2} ,0].
  \end{align}
\end{lemma}

Lemma \ref{Lemma: absorb in V} shows that
there exists a bounded random ball in $\dot{H}^1$
which absorbs any bounded non-random subset of $H$.
Since $\dot{H}^1$ is compactly embedded in $H$,
we have establish the existence of a compact random absorbing set in $H$.
We now state the final theorem and
the proof is similar to that of Theorem 5.6 in \cite{Crauel Flandoli}.

\begin{theorem}
  If $c_0 C_1^2 < \frac{1}{\beta_D(4)  \zeta(4)} $,
  then the random dynamical system associated with \eqref{Equation: integral form}
  has a random attractor.
\end{theorem}


\begin{thebibliography}{99}



\bibitem{Bellout Bloom Necas}H. Bellout, F. Bloom, J. Ne\v{c}as,
Phenomenological behavior of multipolar viscous fluids,
{\it Quart. Appl. Math.} L (1992) 559-583.

\bibitem{Ladyzhenskaya viscousflow}O.A. Ladyzhenskaya,
The Mathematical Theory of Viscous Incompressible Flow,
Gordon and Breach, New York, 1963.

\bibitem{Bloom Hao existence}F. Bloom, W. Hao,
Regularization of a non-Newtonian system in an unbound channel: existence and uniqueness of solutions,
{\it Nonlinear Anal.} 44 (2001) 281-309.

\bibitem{Zhao Duan}C.D. Zhao, J.Q. Duan,
Random attractor for the Ladyzhenskaya modek with additive noise,
J. Math. Anal. Appl. 362 (2010) 241-251.

\bibitem{Zhao Zhou}C.D. Zhao, S.F. Zhou,
Pullback attractors for a non-autonomous incompressible non-Newtonian fluid,
{\it J. Differential Equations}, {\bf 238} (2007) 394-425.

\bibitem{GuoGuo}B. Guo, C. Guo,
The convergence of non-Newtonian fluids to Navier-Stokes equations,
{\it J. Math. Anal. Appl.}, {\bf 35} (2009) 468-478.


\bibitem{Mandelbrot VanNess}B.B. Mandelbrot, J.W. Van Ness,
Fractional Brownian motion , fractional noises and applications,
SIAM Rev. 10 (1968) 422-437.

\bibitem{Li Huang}J. Li, J. Huang,
Dynamics of 2D stochastic non-Newtonian fluids driven by fractional Brownian motion,
http://arxiv.org/abs/1105.2856

\bibitem{Tindel Tudor Viens}S. Tindel, C.A. Tudor, F. Veins,
Stochastic evolution equations with fractional Brownian motion,
{\it Probab. Theory Relat. Fields} 127 (2003) 186-204.

\bibitem{Crauel Flandoli}H. Crauel, F. Flandoli,
Attractors for random dynamical systems,
{\it Probab. Theory Related Fields}, 100 (1994) 465-393.

\bibitem{Lions Magenes}J.L. Lions, E. Magenes,
{\it Non-Homogeneous Boundary Value Problems and Applications Volume I},
Springer-Verlag, New York, 1972.

\bibitem{Chueshov}I.D. Chueshov,
Introduction to the Theory of Infinite-Dimensional Dissipative Systems,
 Acta, 2002.

\bibitem{DanielHenry}D. Henry,
{\it Geometric Theory of Semilinear Parabolic Equations},
Chinese translation, Springer-Verlag, Berlin, 1981.

\bibitem{RogerTemam}R. Temam,
Infinite Dimensional Dynamical Systems in Mechanics and Physics, second ed.,
Springer, New York, 1997.

\bibitem{Biagini Hu Oksendal Zhang}F. Biagini, Y. Hu, B. \O ksendal, T. Zhang,
{\it Stochastic Calculus for Fractional Brownian Motion and Applications},
Springer-Verlag, London, 2008.

\bibitem{Alos Mazet Nualart}E. Alos, O. Mazet, D. Nualart,
Stochastic calculus with respect to Gaussian processes,
{\it Ann. Probab.} 29 (1999) 766-801.

\bibitem{Duncan Maslowski Duncan}T.E. Duncan, B. Maslowski, B.P. Duncan,
Fractional Brownian motion and stochastic equations in Hilbert spaces,
{\it Stochastic Dyn}, {\bf 2} (2002) 225-250.


\bibitem{Duncan Maslowski PasikDuncan}T.E. Duncan, B. Maslowski, B. Pasik-Duncan,
Semilinear stochastic equations in a Hilbert space with a fractional Brownian motion,
{\it SIAM J. MATH. ANAL.} 40 (2009) 2286-2315.

\bibitem{Courant Hilbert}R. Courant, D. Hilbert,
{\it Methods of Mathematical Physics}, Wiley-Interscience, 1953.

\bibitem{Kelliher}J.P. Kelliher,
Eigenvalues of the Stokes operator versus the Dirichlet Laplacian for a bounded domain in the plane,
Pacific Journal of Mathematics, 244(1): 99-132, 2010.

\bibitem{Wang Zeng Guo}G. Wang, M. Zeng, B. Guo,
Stochastic Burgers' equation driven by fractional Brownian motion,
{\it J. Math. Anal. Appl.} 371 (2010) 210-222.





\bibitem{Decreusefond Ustunel}L. Decreusefond, A.S. Ustunel,
Stochastic analysis of the fractional Brownian motion,
{\it Potential Anal.} 10 (1997) 177-214.

\bibitem{Borwein Borwein}J.M. Borwein, P.B. Borwein,
{\it Pi and the AGM: A Study in Analytic Number Theory and Computational Complexity},
New York, Wiley, 1987.





\bibitem{DaPrato Zabczyk}G. Da Prato, J. Zabczyk,
{\it Ergodicity for Infinite Dimesional Systems},
Cambridge University Press, Cambridge, 1996.


\bibitem{Maslowski Schmalfus}B. Maslowski, B. Schmalfu\ss ,
Random dynamics systems and stationary solutions of differential equations driven by
the fractional Brownian motion,
{\it Stoahstic Anal. Appl.}, {\bf 22} (2004), 1557-1607.


\bibitem{JamesRobinson}J.C. Robinson,
Infinite-Dimensional Dynamical Systems, Cambridge University Press, 2001.







%


%
%

\end{thebibliography}
\end{document}